\documentclass{article}
\usepackage{amssymb,amsmath,amsthm,graphicx}

\def\qed{\hfill {\hbox{${\vcenter{\vbox{               
   \hrule height 0.4pt\hbox{\vrule width 0.4pt height 6pt
   \kern5pt\vrule width 0.4pt}\hrule height 0.4pt}}}$}}}

\def\utr{\underline{\triangleright}}
\def\otr{\overline{\triangleright}}

\newtheorem{theorem}{Theorem}

\newtheorem{proposition}[theorem]{Proposition}
\newtheorem{corollary}[theorem]{Corollary}

\theoremstyle{definition}
\newtheorem{example}{Example}
\newtheorem{definition}{Definition}

\date{}

\title{\Large \textbf{Categorification of Biquandle Arrow Weight Invariants via Quivers}}

\author{Sam Nelson\footnote{Email: Sam.Nelson@cmc.edu. Partially supported by Simons Foundation collaboration grant 702597.}\and
Migiwa Sakurai\footnote{Email: migiwa@shibaura-it.ac.jp}}

\begin{document}
\maketitle

\begin{abstract}
Introduced in \cite{NS}, biquandle arrow weights invariants are enhancements 
of the biquandle counting invariant for oriented virtual and classical
knots defined from biquandle-colored Gauss diagrams using a tensor over an 
abelian group satisfying certain properties.
In this paper we  categorify the biquandle arrow weight polynomial invariant 
using biquandle coloring quivers, obtaining new infinite families of 
polynomial invariants of oriented virtual and classical knots. 
\end{abstract}

\parbox{6in} {\textsc{Keywords:} Biquandles, homsets, enhancements, 
categorification, virtual knots, Gauss diagrams, biquandle arrow weights, 
quivers

\smallskip

}

\section{\large\textbf{Introduction}}\label{I}

\textit{Biquandles} are algebraic structures encoding the Reidemeister
moves in knot theory. Associated to every oriented virtual
knot $K$ is a \textit{fundamental biquandle} $\mathcal{B}(K)$ which is 
conjectured to be a complete invariant up to a type of reversed mirror 
image \cite{EN}. Given a finite biquandle $X$, the set of biquandle 
homomorphisms from $\mathcal{B}(K)$ to $X$, $\mathrm{Hom}(\mathcal{B}(K),X)$,
is an invariant of classical and virtual knots known as the \textit{biquandle
homset invariant}. Elements of the homset can be represented as virtual
knot diagrams with elements of $X$ assigned to the semiarcs or as Gauss
diagrams with elements of $X$ assigned to the segments of the circle between
arrows satisfying the biquandle coloring condition; these $X$-colored Gauss
diagrams are analogous to matrices representing linear transformations with 
respect to a choice of basis.

In \cite{NS} the authors introduced an infinite family of enhancements of 
the biquandle homset invariant via structures called \textit{biquandle arrow 
weights} which are analogous to biquandle cocycles arising not from crossings 
in a virtual knot diagram but from arrow crossings in $X$-colored Gauss 
diagrams.

In \cite{CN} the first author introduced a categorification of the quandle
homset invariant known as the \textit{quandle colorings quiver} associated 
to each subset of the set of endomorphisms of the coloring quandle $X$. In 
\cite{JCSN2,CN2,IN} this idea was applied to categorify other homset 
enhancements including
the quandle cocycle enhancement, the quandle module enhancement, the psyquandle
homset invariant and more. From these categorifications, new polynomial 
invariants were defined via decategorification including two-variable 
polynomial invariants and enhancements of the in-degree polynomial 
decategorification.

In this paper we apply the same idea to categorify the biquandle arrow weight
invariant, in the process obtaining new polynomial invariants of virtual knots
via decategorification. The paper is organized as follows. In Section \ref{R}
we briefly review the basics of biquandles, Gauss diagrams and arrow weights. 
In Section \ref{BAWQ} we introduce biquandle arrow weight quivers and obtain
some new polynomial invariants of classical and virtual knots via
decategorification. We observe that this is an infinite family of invariants
with several choices of parameters available including target biquandle $X$,
abelian group $A$, biquandle arrow wight $W$ and subset $S$ of the
the set of endomorphsims of $X$. In Section \ref{E} we collect some examples 
to illustrate the computation of the new invariants and make tables of the 
invariant values for a sampling of virtual knots for some choices of 
parameters $(X,A,W,S)$. We conclude in Section \ref{Q} with some questions
for future research.

This paper, including all text, illustrations, and python code for 
computations, was written strictly by the authors without the use of 
generative AI in any form.

\section{\large\textbf{Review of Biquandles, Gauss Diagrams and Arrow Weights}}\label{R}

In order to introduce our new categorification of the Biquandle Arrow Weight 
invariant, it will be useful to briefly recall the orginal invariant. More 
detail can be found in our earlier paper \cite{NS}.

%
%

A \textit{biquandle} is a set $X$ with two binary operations $\utr,\otr$
satisying certain axioms (see \cite{EN} or \cite{NS}) 
derived from the Reidemeister moves. More precisely,
a \textit{biquandle coloring} or an \textit{$X$-coloring} of an oriented knot 
or link diagram $D$ by a biquandle $X$ assigns an element of 
an element of $X$ to each semiarc in $D$ so that at every crossing 
we have the condition
\[\scalebox{0.3}{\includegraphics{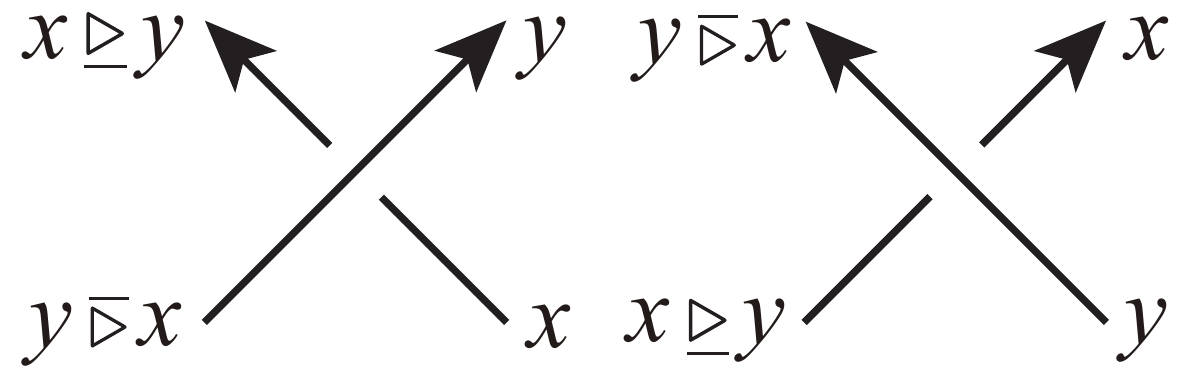}},\]
called the \textit{biquandle coloring condition}.

The biquandle axioms are then the conditions required such that given any
biquandle coloring of a diagram before a Reidemeister move, there exists 
a unique coloring of the diagram after the move which agrees with the 
original coloring outside the neighborhood of the move. Then, it follows 
that the number of colorings of an oriented knot or link diagram by a 
finite biquandle $X$ is a (possibly zero) integer-valued invariant, 
denoted $\Phi_X^{\mathbb{Z}}$, called the \textit{biquandle counting invariant}.

In particular, associated to every oriented classical or virtual knot $K$ 
there is a \textit{fundamental biquandle} $\mathcal{B}(K)$ defined 
from any diagram of $L$ with a presentation given by generators 
associated to semiarcs and relations at the crossings. An $X$-coloring 
then determines and is determined by a unique \textit{biquandle homomorphism}
$f:\mathcal{B}(K)\to X$. Then the sets $\mathrm{Hom}(\mathcal{B}(K),X)$
of biquandle homomorphisms (known as the \textit{biquandle homset}) is an 
invariant of $K$. More precisely, choosing a different diagram for $K$ changes 
the way each homset element is represented, with $X$-colored diagrams 
representing the same homset element iff they are related by $X$-colored
Reidemeister moves. Choosing a diagram for $K$ is analogous to choosing input 
and output bases for vector spaces in linear algebra, with $X$-colored diagrams
representing homset elements analogously to matrices representing linear 
transformations with Reidemeister moves playing the role of change of basis 
matrices.

Let $K$ be a virtual knot and $D$ a virtual knot diagram of an oriented
virtual knot $K$.  
Then, $D$ may be regarded as the image $f(\mathbb{S}^1)$ of a generic immersion 
$f : \mathbb{S}^1$ $\to$ $\mathbb{R}^2$.  
Recall that
a \textit{Gauss diagram} for $D$ is 
the preimage of $D$ with chords, each of which connects the preimages of 
each real crossing. 
We can
specify crossing information of each real crossing on the corresponding 
chord by directing the chord toward the undercrossing point and assigning 
each chord with the sign (i.e., the local writhe) of the crossing as shown: 
\[\includegraphics[width=4.5cm,clip]{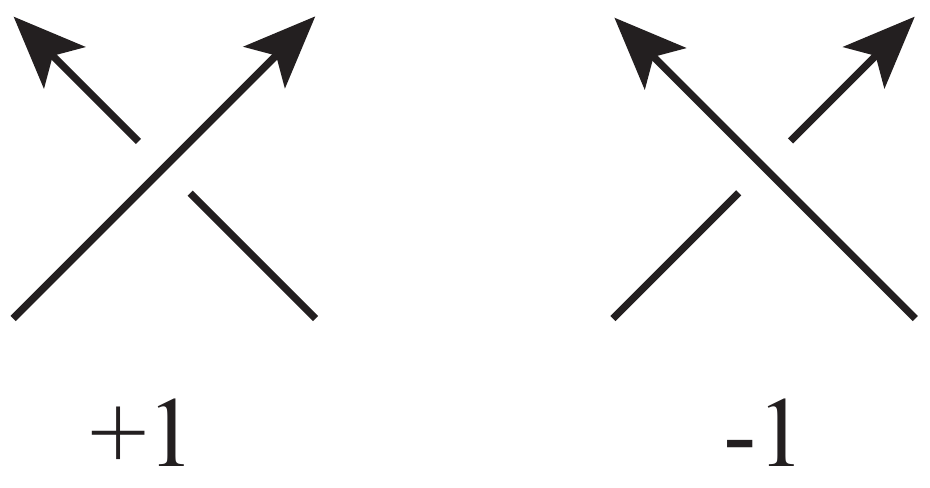}\]


A \textit{biquandle coloring} of a Gauss diagram by a biquandle $X$ is then an 
assignment of an element of $X$ to each segment of the circle between the 
arrowheads and tails such that at every arrow we have
as shown:
\[\scalebox{0.6}{\includegraphics{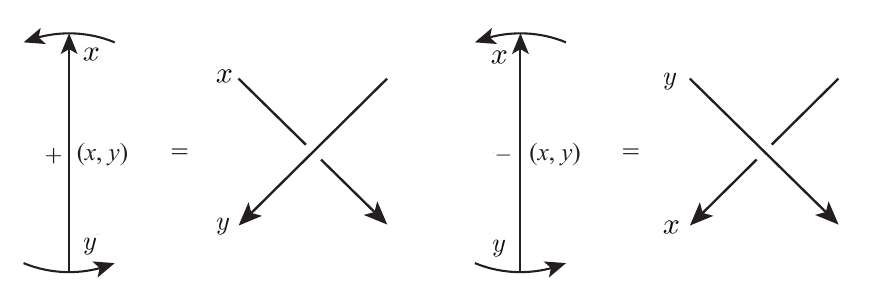}}\]
In particular, to each arrow in a biquandle-colored Gauss diagram
we associate an ordered pair $(x,y)$ of biquandle elements together with a 
sign. The set of biquandle colorings of a Gauss diagram
$D$ representing a virtual knot $K$ by a biquandle $X$ can be identified 
with the homset $\mathrm{Hom}(\mathcal{B}(K),X)$ with each $X$-coloring 
representing a homset element and with a choice of diagram $D$ being 
analogous to a choice of basis for a vector space.

%

Finally, recall from \cite{NS} that a \textit{biquandle arrow weight} is a 
function $\phi$ sending pairs of pairs of biquandle elements to elements in
an abelian group $A$ satisfying axioms which can be found in \cite{NS}.
The biquandle arrow weight axioms are chosen so that the sum $\Sigma_D$ of 
$\phi((a,b),(c,d))$ values over the set of arrow crossings with arrows 
labeled $(a,b)$ and $(c,d)$ and sign given by the product of arrow
signs over the set of arrows in an $X$-colored Gauss diagram 
\[\includegraphics{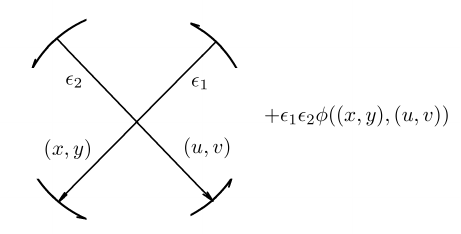}\]
is unchanged by Gauss diagram Reidemeister moves.

Given a finite biquandle $X$ and biquandle arrow weight $W$ with values in 
an abelian group $A$, for any oriented classical or virtual knot $K$ the
multiset of  values over the set of $X$-colorings of
any Gauss diagram $D$ representing $K$ is an invariant of knots, known
as the \textit{biquandle arrow weight multiset}, denoted $\Phi_X^{W,M}(K)$. 
We can convert this multiset into a ``polynomial'' form by summing over 
the multiset a formal variable raised to the ``power'' of each multiset 
element, obtaining the \textit{biquandle arrow weight polynomial} 
$\Phi_X^{W}(K)$. We note that depending on $A$, the ``polynomial'' may be
regarded as a bookkeeping device rather than a true polynomial.

We then have:

\begin{theorem} (\cite{NS})
The biquandle arrow weight multiset and polynomial are invariants of classical
and virtual knots.
\end{theorem}

\begin{example}\label{Ex1}
Let $X$ be the biquandle specified by the operation tables
\[
\begin{array}{r|rrr}
\utr & 1 & 2 \\ \hline
1 & 2 & 2 \\
2 & 1 & 1
\end{array}
\quad
\begin{array}{r|rrr}
\otr & 1 & 2 \\ \hline
1 & 2 & 2 \\
2 & 1 & 1
\end{array} 
\]
and let $A=\mathbb{Z}_{16}$. Then we compute via \texttt{python} that we have
biquandle arrow weight given by
\[
\left[\begin{array}{cc}
\left[\begin{array}{cc}0 & 4 \\ 4 & 0 \end{array}\right] &
\left[\begin{array}{cc}4 & 0 \\ 8 & 12\end{array}\right] \\
& \\
\left[\begin{array}{cc}4 & 8 \\ 0 & 12\end{array}\right] &
\left[\begin{array}{cc}0 & 12 \\ 12 & 0\end{array}\right]
\end{array}\right]
\]
among others. The virtual trefoil knot $2.1$ then has two $X$-colorings
\[\includegraphics{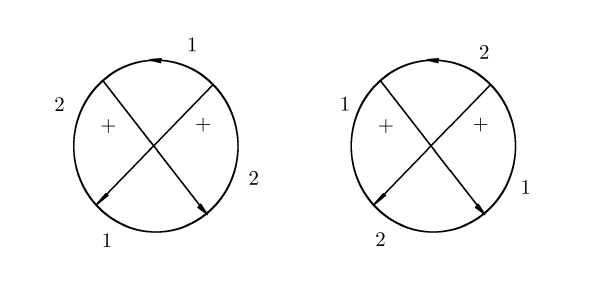}\]
yielding biquandle arrow weight multiset $\Phi_X^{W,M}(2.1)=\{8,8\}$ and
polynomial $2u^8$.
\end{example}

\section{\Large\textbf{Biquandle Arrow Weight Quivers}}\label{BAWQ}

We can now introduce our infinite family of categorfications of the biquandle
arrow weight invariants. 

Let $K$ be an oriented virtual knot represented by a Gauss diagram $D$
(see \cite{EN,GPV}).
Let $X$ be a finite biquandle, $A$ an abelian group and $\beta$
a biquandle arrow weight for $X$ with coefficients in $A$. 

Recall that a map $\phi:X\to X$ is a \textit{biquandle endomorphism} if
for all $x,x'\in X$ we have
\[\phi(x\, \utr\, x')=\phi(x)\, \utr\, \phi(x') \ \mathrm{and} \ 
\phi(x\, \otr\,  x')=\phi(x)\, \otr\, \phi(x').\]
Given an $X$-coloring of $D$, if we apply $\phi$ to all of the semiarc colors
the result is also an $X$-coloring of $D$.
\[\includegraphics{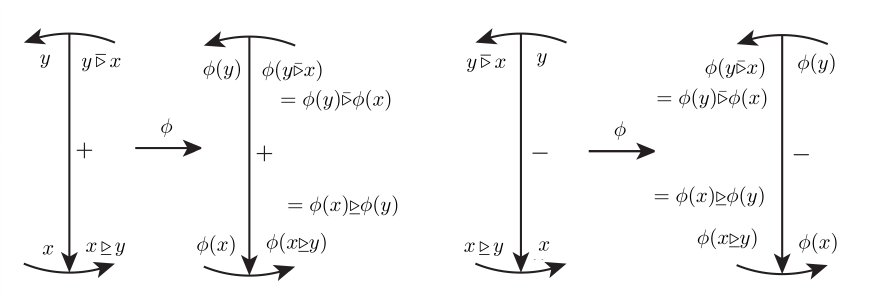}\]

It then follows, as in \cite{CN}, that the directed graph or \textit{quiver}
with a vertex for each $X$-coloring of $D$ and directed edge connecting
each coloring to its image under each $\phi\in S$ is unchanged by 
Reidemeister moves and hence defines an invariant of virtual knots, which 
we call the \textit{biquandle coloring quiver} $BCB_{X,S}(K)$. When 
$S=\mathrm{Hom}(X,X)$ is the complete set of endomorphisms, we say that
$BCB_{X,S}(K)$ is the \textit{full quiver}.

Since the vertices in the quiver represent the elements of the biquandle
homset, given a biquandle arrow weight $W$ we can associate to each
vertex in the quiver the corresponding $\Sigma_D$ value to obtain a 
\textit{weighted quiver} $\mathcal{Q}_{X,W,S}(K)$, i.e. a quiver weighted
with elements of $A$ at the vertices. We then have our main result:

\begin{proposition}
The weighted quiver $\mathcal{Q}_{X,W,S}(K)$ is an invariant of oriented 
classical and virtual knots.
\end{proposition}

\begin{proof}
All of the ingredients are already invariant under classical and virtual
Reidemeister moves.
\end{proof}

Quivers are categories with vertices as objects and directed paths as morphisms.
Thus, this weighted quiver categorifies the biquandle arrow weight polynomial,
which we can now regard as a decategorification by summing over the vertices
terms of the form $u^{w(v)}$ where $w(v)$ is the weight at each vertex.

We additionally define two standard polynomial decategorifications,
yielding new polynomial invariants of classical and virtual knots.

\begin{definition}
Let $X$ be a finite biquandle, $S\subset \mathrm{Hom}(X,X)$ a set of 
biquandle endomorphisms, $A$ an abelian group and $W$ a biquandle arrow
weight with $A$ coefficients. Then for any oriented classical or virtual 
knot or link $K$ represented by a Gauss diagram $D$, we define 
\begin{itemize}
\item \textit{The arrow weight in-degree polynomial} to be the sum, over all
vertices $v$ in the vertex set $V$ of the quiver, 
of products of a formal variable $u$ to the power of the vertex weight 
$w(v)=\Sigma_D$ at the vertex $v$ times a formal variable $w$ to the power
of the in-degree of the vertex $v$ (the number of arrows directed in to $v$),  
\[\Phi_{X,W,S}^{\mathrm{deg}_+}(K)=\sum_{v\in V} u^{w(v)}w^{\deg_+(v)}\]
and we further define
\item \textit{The arrow weight two-variable polynomial} to be the
sum, over all edges $e$ in the edge set $E$ of the quiver, of products of 
a formal variable $s$ to the power of the vertex weight at the source 
vertex $S(e)$ times a formal variable $t$ to the power of the vertex weight
at the target vertex $T(e)$,
\[\Phi_{X,W,S}^{2}(K)=\sum_{e\in E} s^{S(e)}t^{T(e)}.\]
\end{itemize}
\end{definition}

\begin{proposition}
The polynomials $\Phi_{X,W,S}^{\mathrm{deg}_+}(K)$ and $\Phi_{X,W,S}^{2}(K)$ are 
invariants of oriented classical and virtual knots.
\end{proposition}

\begin{proof}
These are determined by the biquandle arrow weighted quiver, which is 
invariant.
\end{proof}

We can also make use of the weighting of the vertices to define a new invariant
quiver-valued invariant from which we can extract new polynomial invariants
as decategorifications. 

\begin{definition}
Let $X$ be a biquandle, $W$ a biquandle arrow weight with coefficients in
an abelian group $A$, and $S\subset \mathrm{Hom}(X,X)$ a set of biquandle
endomorphisms. Then let us form the quotient $\mathcal{Q}^Q_{X,W,S}(K)$ 
of the weighted quiver $\mathcal{Q}_{X,W,S}(K)$ by identifying the sets of
vertices which share a common weight. We will call this the \textit{arrow 
weighted quotient quiver}.
\end{definition}

\begin{proposition}
The arrow weighted quotient quiver $\mathcal{Q}^Q_{X,W,S}(K)$ is an invariant
of classical and virtual knots.
\end{proposition}

\begin{proof}
The arrow weighted quiver $\mathcal{Q}_{X,W,S}(K)$ is already unchanged by 
Reidemeister moves before taking the quotient; it follows that the quotient
is also unchanged by Reidemeister moves.
\end{proof}

One easy polynomial invariant we can obtain from $\mathcal{Q}^Q_{X,W,S}(K)$
is defined by letting each vertex contribute its weight and number of loops:

\begin{definition}
Let $X$ be a biquandle, $W$ a biquandle arrow weight with coefficients in
an abelian group $A$, and $S\subset \mathrm{Hom}(X,X)$ a set of biquandle
endomorphisms. Let $\mathcal{Q}^Q_{X,W,S}(K)$ be the resulting arrow weighted
quotient quiver, and let its vertex set be $V$. Then the \textit{arrow weighted
quotient quiver loop polynomial} is
\[\Phi_{X,W,S}^{Q,L}(K)=\sum_{v\in V} L(v)x^{W(v)}\]
where $W(v)$ is the biquandle arrow weight at the vertex $v$ and $L(v)$ is 
the number of loops in the quiver based at $v$.
\end{definition}

\begin{corollary}
The arrow weighted quotient quiver loop polynomial $\Phi_{X,W,S}^{Q,L}(K)$ 
is an invariant of classical and virtual knots.
\end{corollary}

\section{\Large\textbf{Examples}}\label{E}

In this section we compute some examples of biquandle arrow weight quivers 
and their decategorification polynomial invariants. We remark that these
are toy examples in the sense that the true power of this infinite family
of invariants lies in using larger biquandles and coefficient groups than 
we are currently able to compute.

\begin{example}
Let $X$, $A$ and $W$ be the biquandle and biquandle arrow weight from Example
\ref{Ex1}. Then the virtual knot $2.1$ has biquandle arrow weight quiver
\[\includegraphics{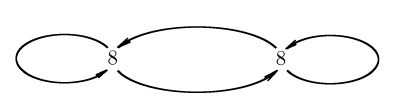}\]
with respect to the full endomorphism set $S=\{[1,2],[2,1]\}$ where
we specify an endomorphism $\sigma:X\to X$ where $X=\{x_1,\dots, x_n\}$
with a vector $\sigma=[\sigma(x_1),\dots,\sigma(x_n)]$ of the images of 
elements of $X$. 
\end{example}

\begin{example}
Let $X$ be the biquandle given by the operation tables
\[
\begin{array}{r|rrr}
\utr & 1 & 2 & 3 \\ \hline
1 & 3 & 3 & 3 \\
2 & 1 & 1 & 1 \\
3 & 2 & 2 & 2
\end{array}
\quad
\begin{array}{r|rrr}
\otr & 1 & 2 & 3  \\ \hline
1 & 3 & 3 & 3 \\
2 & 1 & 1 & 1 \\
3 & 2 & 2 & 2
\end{array}
\]
with endomorphism set $S=\{[3\ 1\ 2],[1\ 2\ 3],[2\ 3\ 1]\}$. 
We then compute that the 4-tensor
\[\left[\begin{array}{ccc}
\left[\begin{array}{rrr}
0 & 1 & 3 \\ 3 & 0 & 1 \\ 1 & 3 & 0
\end{array}\right] & \left[\begin{array}{rrr}
1 & 0 & 4 \\ 4 & 1 & 0 \\ 0 & 4 & 1
\end{array}\right] & \left[\begin{array}{rrr}
3 & 4 & 0 \\ 0 & 3 & 4 \\ 4 & 0 & 3
\end{array}\right] \\ \left[\begin{array}{rrr}
3 & 4 & 0 \\ 0 & 3 & 4 \\ 4 & 0 & 3
\end{array}\right] & \left[\begin{array}{rrr}
0 & 1 & 3 \\ 3 & 0 & 1 \\ 1 & 3 & 0
\end{array}\right] & \left[\begin{array}{rrr}
1 & 0 & 4 \\ 4 & 1 & 0 \\ 0 & 4 & 1
\end{array}\right] \\ \left[\begin{array}{rrr}
1 & 0 & 4 \\ 4 & 1 & 0  \\ 0 & 4 & 1
\end{array}\right] & \left[\begin{array}{rrr}
3 & 4 & 0 \\ 0 & 3 & 4 \\ 4 & 0 & 3
\end{array}\right] & \left[\begin{array}{rrr}
0 & 1 & 3 \\ 3 & 0 & 1 \\ 1 & 3 & 0 
\end{array}\right]
\end{array}\right]
\]
defines a biquandle arrow weight with coefficients in $A=\mathbb{Z}_8$.

Computing the in-degree polynomials, we obtain the table
\[
\begin{array}{r|l}
\Phi_{X,W,S}^{\mathrm{deg}_+}(K) & K \\ \hline
3w^3 & 3.3, 3.4, 3.5, 3.6, 3.7, 4.1, 4.2, 4.3, 4.6, 4.7, 4.8, 4.12, 4.16, 4.21, 4.25, 4.31, 4.36, 4.43, 4.46, \\ 
&  4.53, 4.55, 4.56, 4.59, 4.65, 4.71, 4.73, 4.75, 4.76, 4.77, 4.85, 4.86, 4.89, 4.90, 4.92, 4.95,\\ 
& 4.96, 4.98, 4.99, 4.100, 4.101, 4.104, 4.105, 4.106, 4.107, 4.108 \\
3uw^3 &4.22, 4.62, 4.83, 4.87, 4.88\\
3u^2w^3 & 4.13, 4.15, 4.17, 4.20, 4.23, 4.29, 4.32, 4.34, 4.38, 4.41, 4.49, 4.51, 4.58, 4.61, 4.67, 4.72 \\
3u^3w^3 & 4.24, 4.26, 4.42, 4.47, 4.80 \\
3u^4w^3 & 2.1, 3.1, 3.2, 4.4, 4.5, 4.9, 4.14, 4.27, 4.30, 4.37, 4.44, 4.48, 4.52, 4.54, 4.60, 4.64, 4.69, 4.74, \\ 
& 4.82, 4.84, 4.91, 4.94, 4.102 \\
3u^5w^3 & 4.66, 4.79, 4.93 \\
3u^6w^3 & 4.10, 4.11, 4.18, 4.19, 4.33, 4.35, 4.40, 4.50, 4.57, 4.63, 4.68, 4.70 \\
3u^7w^3 & 4.28, 4.39, 4.45, 4.78, 4.81, 4.97, 4.103
\end{array}.
\]
Since all of these virtual knots have the same unweighted in-degree 
polynomial, $3w^3$, this example shows that our new construction is a proper
enhancement of the in-degree polynomial.


\end{example}

\begin{example}\label{Ex3}
Let $X$ be the biquandle specified by the operation tables
\[
\begin{array}{r|rrrr}
\utr & 1 & 2 & 3 & 4 \\ \hline
1 & 1 & 3 & 4 & 2 \\
2 & 2 & 4 & 3 & 1 \\
3 & 3 & 1 & 2 & 4 \\
4 & 4 & 2 & 1 & 3
\end{array}
\quad
\begin{array}{r|rrrr}
\otr & 1 & 2 & 3 & 4 \\ \hline
1 & 1 & 1 & 1 & 1 \\
2 & 4 & 4 & 4 & 4 \\
3 & 2 & 2 & 2 & 2 \\
4 & 3 & 3 & 3 & 3
\end{array}.
\]
Let $S=\{[1,4,2,3],[1,1,1,1],[1,3,4,2],[1,2,3,4]\}\subset \mathrm{Hom}(X,X)$ and
let $W$ be the biquandle arrow weight with coefficient group $A=\mathbb{Z}_6$
specified by the 4-tensor
\[
\left[\begin{array}{cccc}
\left[\begin{array}{rrrr}
0 & 3 & 0 & 3 \\ 0 & 3 & 0 & 3 \\ 0 & 3 & 0 & 3 \\ 0 & 3 & 0 & 3  
\end{array}\right] & \left[\begin{array}{rrrr}
3 & 0 & 0 & 3 \\ 3 & 0 & 0 & 3 \\ 3 & 0 & 0 & 3 \\ 3 & 0 & 0 & 3  
\end{array}\right] & \left[\begin{array}{rrrr}
0 & 0 & 0 & 0 \\ 0 & 0 & 0 & 0 \\ 0 & 0 & 0 & 0 \\ 0 & 0 & 0 & 0 
\end{array}\right] & \left[\begin{array}{rrrr}
3 & 3 & 0 & 0 \\ 3 & 3 & 0 & 0 \\ 3 & 3 & 0 & 0 \\ 3 & 3 & 0 & 0
\end{array}\right] \\ 
\left[\begin{array}{rrrr}
0 & 3 & 0 & 3 \\ 0 & 3 & 0 & 3 \\ 0 & 3 & 0 & 3 \\ 0 & 3 & 0 & 3
\end{array}\right] & \left[\begin{array}{rrrr}
3 & 0 & 0 & 3 \\ 3 & 0 & 0 & 3 \\ 3 & 0 & 0 & 3 \\ 3 & 0 & 0 & 3
\end{array}\right] & \left[\begin{array}{rrrr}
0 & 0 & 0 & 0 \\ 0 & 0 & 0 & 0 \\ 0 & 0 & 0 & 0 \\ 0 & 0 & 0 & 0
\end{array}\right] & \left[\begin{array}{rrrr}
3 & 3 & 0 & 0 \\ 3 & 3 & 0 & 0 \\ 3 & 3 & 0 & 0 \\ 3 & 3 & 0 & 0
\end{array}\right] \\
\left[\begin{array}{rrrr}
0 & 3 & 0 & 3 \\ 0 & 3 & 0 & 3 \\ 0 & 3 & 0 & 3 \\ 0 & 3 & 0 & 3
\end{array}\right] & \left[\begin{array}{rrrr}
3 & 0 & 0 & 3 \\ 3 & 0 & 0 & 3 \\ 3 & 0 & 0 & 3 \\ 3 & 0 & 0 & 3
\end{array}\right] & \left[\begin{array}{rrrr}
0 & 0 & 0 & 0 \\ 0 & 0 & 0 & 0 \\ 0 & 0 & 0 & 0 \\ 0 & 0 & 0 & 0
\end{array}\right] & \left[\begin{array}{rrrr}
3 & 3 & 0 & 0 \\ 3 & 3 & 0 & 0 \\ 3 & 3 & 0 & 0 \\ 3 & 3 & 0 & 0
\end{array}\right] \\
\left[\begin{array}{rrrr}
0 & 3 & 0 & 3 \\ 0 & 3 & 0 & 3 \\ 0 & 3 & 0 & 3 \\ 0 & 3 & 0 & 3
\end{array}\right] & \left[\begin{array}{rrrr}
3 & 0 & 0 & 3 \\ 3 & 0 & 0 & 3 \\ 3 & 0 & 0 & 3 \\ 3 & 0 & 0 & 3
\end{array}\right] & \left[\begin{array}{rrrr}
0 & 0 & 0 & 0 \\ 0 & 0 & 0 & 0 \\ 0 & 0 & 0 & 0 \\ 0 & 0 & 0 & 0
\end{array}\right] & \left[\begin{array}{rrrr}
3 & 3 & 0 & 0 \\ 3 & 3 & 0 & 0 \\ 3 & 3 & 0 & 0 \\ 3 & 3 & 0 & 0
\end{array}\right]
\end{array}\right]
\]
Then the virtual knots $3.1$ and $3.3$ have biquandle arrow weight quivers
\[\includegraphics{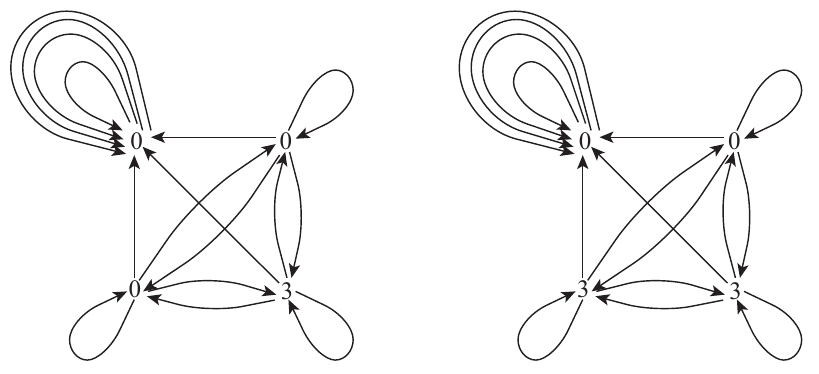} \]
with respect to $X,S$ and $W$. In particular we note that the coloring quivers
without the arrow weights are isomorphic, showing that the biquandle arrow
weight quiver is a proper enhancement of the biquandle coloring quiver.
\end{example}

\begin{example}
Continuing with the virtual knots, biquandle, arrow weight and endomorphism set
from Example \ref{Ex3}, we obtain the quotient quivers
\[\includegraphics{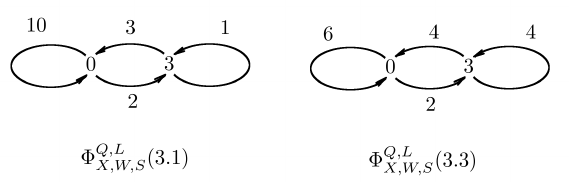}\]
and quotient quiver loop polynomials $\Phi_{X,W,S}^{Q,L}(3.1)= 10x^0+1x^3$
and $\Phi_{X,W,S}^{Q,L}(3.3)= 6x^0+4x^3$.
\end{example}

\begin{example}
We computed the biquandle arrow weight 2-variable
polynomials for a choice of orientation for each of the virtual knots with 
up to 4 classical crossings at \cite{KA} for the biquandle with operation
tables
\[
\begin{array}{r|rrr}
\utr & 1 & 2 & 3 \\ \hline
1 & 3 & 3 & 3 \\
2 & 1 & 1 & 1 \\
3 & 2 & 2 & 3 \\
\end{array}
\quad
\begin{array}{r|rrr}
\otr & 1 & 2 & 3 \\ \hline
1 & 3 & 3 & 3 \\
2 & 1 & 1 & 1 \\
3 & 2 & 2 & 3 \\
\end{array},
\] 
arrow weight over $A=\mathbb{Z}_3$ given by the operation tensor
\[
\left[\begin{array}{ccc}
\left[\begin{array}{rrr} 0 & 1 & 2 \\ 2 & 0 & 1 \\ 1 & 2 & 0\end{array}\right] & 
\left[\begin{array}{rrr} 1 & 0 & 0 \\ 0 & 1 & 0 \\ 0 & 0 & 1\end{array}\right] & 
\left[\begin{array}{rrr} 2 & 0 & 0 \\ 0 & 2 & 0 \\ 0 & 0 & 2\end{array}\right] \\
\left[\begin{array}{rrr} 2 & 0 & 0 \\ 0 & 2 & 0 \\ 0 & 0 & 2\end{array}\right] & 
\left[\begin{array}{rrr} 0 & 1 & 2 \\ 2 & 0 & 1 \\ 1 & 2 & 0\end{array}\right] & 
\left[\begin{array}{rrr} 1 & 0 & 0 \\ 0 & 1 & 0 \\ 0 & 0 & 1\end{array}\right] \\
\left[\begin{array}{rrr} 1 & 0 & 0 \\ 0 & 1 & 0 \\ 0 & 0 & 1\end{array}\right] & 
\left[\begin{array}{rrr} 2 & 0 & 0 \\ 0 & 2 & 0 \\ 0 & 0 & 2\end{array}\right] & 
\left[\begin{array}{rrr} 0 & 1 & 2 \\ 2 & 0 & 1 \\ 1 & 2 & 0\end{array}\right] \\
\end{array}\right]
\]
and endomorphism set 
\[S=\{[3,2,1],[1,2,3],[2,3,1]\}.\] The results are in the table.

\[\begin{array}{r|l}
\Phi_{X,W,S}^{2}(K) & K \\ \hline
9 & 2.1, 3.1, 3.2, 3.3, 3.4, 3.5, 3.6, 3.7, 4.1, 4.2, 4.3, 4.4, 4.5, 4.6, 4.7, 4.8, 4.9, 4.12, 4.14, 4.16, 4.21, \\ 
& 4.25, 4.26, 4.27, 4.28, 4.30, 4.31, 4.36, 4.37, 4.43, 4.44, 4.45, 4.46, 4.47, 4.48, 4.52, 4.53, 4.54,\\ 
&  4.55, 4.56, 4.59, 4.60, 4.64, 4.65, 4.69, 4.71, 4.73, 4.74, 4.75, 4.76, 4.77, 4.80, 4.81, 4.82, 4.83, \\ 
& 4.84, 4.85, 4.86, 4.87, 4.88, 4.89, 4.90, 4.91, 4.92, 4.93, 4.94, 4.95, 4.96, 4.97, 4.98, 4.99, 4.100,\\ 
& 4.101, 4.102, 4.103, 4.104, 4.105, 4.106, 4.107, 4.108 \\
9st & 4.11, 4.17, 4.22, 4.23, 4.32, 4.61, 4.62, 4.63, 4.66, 4.67, 4.68, 4.79 \\
9s^2t^2 & 4.10, 4.13, 4.15, 4.18, 4.19, 4.20, 4.24.4.29, 4.33, 4.34, 4.35, 4.38, 4.39, 4.40, 4.41, 4.42, 4.49, \\ 
& 4.50, 4.51, 4.57, 4.58, 4.70, 4.72, 4.78
\end{array}
\]
\end{example}

\begin{example}
We computed the biquandle arrow weight quotient quiver loop
polynomials for a choice of orientation for each of the virtual knots with 
up to 4 classical crossings at \cite{KA} for the biquandle with operation
table
\[
\begin{array}{r|rrrr}
\utr & 1 & 2 & 3 & 4 \\ \hline
1 & 2 & 2 & 2 & 2 \\
2 & 3 & 3 & 3 & 3 \\
3 & 4 & 4 & 4 & 4 \\
4 & 1 & 1 & 1 & 1
\end{array}
\quad
\begin{array}{r|rrrr}
\otr & 1 & 2 & 3 & 4 \\ \hline
1 & 2 & 4 & 2 & 4 \\
2 & 1 & 3 & 1 & 3 \\
3 & 4 & 2 & 4 & 2 \\
4 & 3 & 1 & 3 & 1
\end{array}
\] 
arrow weight over $A=\mathbb{Z}_4$ given by the tensor
\[
\left[\begin{array}{cccc}
\left[\begin{array}{rrrr}
0 & 0 & 0 & 0 \\ 2 & 2 & 2 & 2 \\ 2 & 2 & 2 & 2 \\ 0 & 0 & 0 & 0  
\end{array}\right] & \left[\begin{array}{rrrr}
0 & 0 & 0 & 0  \\ 2 & 2 & 2 & 2 \\ 2 & 2 & 2 & 2 \\ 0 & 0 & 0 & 0  
\end{array}\right] & \left[\begin{array}{rrrr}
0 & 0 & 0 & 0 \\ 2 & 2 & 2 & 2 \\ 2 & 2 & 2 & 2 \\ 0 & 0 & 0 & 0 
\end{array}\right] & \left[\begin{array}{rrrr}
0 & 0 & 0 & 0 \\ 2 & 2 & 2 & 2 \\ 2 & 2 & 2 & 2 \\ 0 & 0 & 0 & 0 
\end{array}\right] \\ \left[\begin{array}{rrrr}
2 & 2 & 2 & 2 \\ 0 & 0 & 0 & 0 \\ 2 & 2 & 2 & 2 \\ 0 & 0 & 0 & 0  
\end{array}\right] & \left[\begin{array}{rrrr}
2 & 2 & 2 & 2 \\ 0 & 0 & 0 & 0 \\ 2 & 2 & 2 & 2 \\ 0 & 0 & 0 & 0  
\end{array}\right] & \left[\begin{array}{rrrr}
2 & 2 & 2 & 2 \\ 0 & 0 & 0 & 0 \\ 2 & 2 & 2 & 2 \\ 0 & 0 & 0 & 0 
\end{array}\right] & \left[\begin{array}{rrrr}
2 & 2 & 2 & 2 \\ 0 & 0 & 0 & 0 \\ 2 & 2 & 2 & 2 \\ 0 & 0 & 0 & 0 
\end{array}\right] \\ \left[\begin{array}{rrrr}
2 & 2 & 2 & 2 \\ 2 & 2 & 2 & 2 \\ 0 & 0 & 0 & 0 \\ 0 & 0 & 0 & 0 
\end{array}\right] & \left[\begin{array}{rrrr}
2 & 2 & 2 & 2 \\ 2 & 2 & 2 & 2 \\ 0 & 0 & 0 & 0 \\ 0 & 0 & 0 & 0  
\end{array}\right] & \left[\begin{array}{rrrr}
2 & 2 & 2 & 2 \\ 2 & 2 & 2 & 2 \\ 0 & 0 & 0 & 0 \\ 0 & 0 & 0 & 0  
\end{array}\right] & \left[\begin{array}{rrrr}
2 & 2 & 2 & 2 \\ 2 & 2 & 2 & 2 \\ 0 & 0 & 0 & 0 \\ 0 & 0 & 0 & 0
\end{array}\right] \\ \left[\begin{array}{rrrr}
0 & 0 & 0 & 0 \\ 0 & 0 & 0 & 0 \\ 0 & 0 & 0 & 0 \\ 0 & 0 & 0 & 0  
\end{array}\right] & \left[\begin{array}{rrrr}
0 & 0 & 0 & 0 \\ 0 & 0 & 0 & 0 \\ 0 & 0 & 0 & 0 \\ 0 & 0 & 0 & 0  
\end{array}\right] & \left[\begin{array}{rrrr}
0 & 0 & 0 & 0 \\ 0 & 0 & 0 & 0 \\ 0 & 0 & 0 & 0 \\ 0 & 0 & 0 & 0  
\end{array}\right] & \left[\begin{array}{rrrr}
0 & 0 & 0 & 0 \\ 0 & 0 & 0 & 0 \\ 0 & 0 & 0 & 0 \\ 0 & 0 & 0 & 0
\end{array}\right] \end{array}\right] 
\]
and endomorphism set 
\[S=\{[4,1,2,3],[1,2,3,4],[2,3,4,1],[3,4,1,2]\}.\] The results are in 
the table.

\[\begin{array}{r|l}
\Phi_{X,W,S}^{Q,L}(K) & K \\ \hline
4x^2 + 4 & 2.1, 3.2, 3.3, 3.4, 4.3, 4.4, 4.5, 4.6, 4.9, 4.11, 4.13, 4.14, 4.15, 4.17, 4.18, 4.19, 4.20, 4.22, 4.23 \\ 
& 4.24, 4.25, 4.26, 4.27, 4.28, 4.29, 4.30, 4.31, 4.33, 4.34, 4.35, 4.37, 4.38, 4.39, 4.40, 4.42, 4.44, \\ 
& 4.45, 4.46, 4.48, 4.49, 4.51, 4.52, 4.54, 4.57, 4.60, 4.61, 4.62, 4.63, 4.64, 4.66, 4.67, 4.69, 4.74, \\ 
& 4.78, 4.79, 4.81, 4.82, 4.83, 4.84, 4.87, 4.88, 4.92, 4.93, 4.94, 4.95, 4.97, 4.101, 4.103, 4.104 \\
16 & 3.1, 3.6, 4.1, 4.8, 4.10, 4.16, 4.32, 4.41, 4.43, 4.47, 4.50, 4.55, 4.56, 4.58, 4.59, 4.68, 4.71, 4.72, \\ 
& 4.73, 4.75, 4.76, 4.77, 4.80, 4.89, 4.90, 4.96, 4.98, 4.99, 4.102, 4.105, 4.107, 4.108 \\
16x^2 &  3.5, 3.7, 4.2, 4.7, 4.12, 4.21, 4.36, 4.53, 4.65, 4.70, 4.85, 4.86, 4.91, 4.100, 4.106
\end{array}
\]
\end{example}

\section{\Large\textbf{Questions}}\label{Q}

As mentioned in Section \ref{E}, faster search and computation algorithms
for finding biquandle arrow weight quivers with respect to biquandles
with larger cardinality and larger finite or infinite coefficient groups
are needed. Our current custom \texttt{python} code can handle finding
``toy'' examples with biquandles of small cardinality and with
small coefficient groups such as $\mathbb{Z}_n$ for small $n$, but the 
true power of this infinite family of invariants awaits faster and more
efficient algorithms for finding biquandle arrow weights and quivers.

What is the relationship between biquandle arrow weight quivers and (bi)quandle
cocycle quivers as defined in \cite{CN2}? Indeed, what is the algebraic structure
in general of biquandle arrow weights, and how are they related to (bi)quandle
cohomology?

What is the geometric meaning of a biquandle arrow weight and its associated
quivers?

What other decategorifications are possible? What about further enhancements of
the weighted quiver structure? How about functors to other categories?

\bibliographystyle{abbrv}
\bibliography{sam-migiwa}

\bigskip

\noindent
\textsc{Department of Mathematical Sciences \\
Claremont McKenna College \\
850 Columbia Ave. \\
Claremont, CA 91711} 

\

\noindent
\textsc{ \\
Graduate School of Engineering and Science,
College of Engineering,\\
Shibaura Institute of Technology\\
307 Fukasaku, Minuma-ku, Saitama-shi, Saitama, 337-8570, Japan
}

\end{document}